\documentclass[12pt]{amsart}
\usepackage{mathrsfs}      
\usepackage{txfonts}
\usepackage{amssymb}
\usepackage{eucal}
\usepackage{graphicx}
\usepackage{amsmath}
\usepackage{amscd}
\usepackage[all]{xy}           
\usepackage[active]{srcltx} 

\usepackage{amsfonts,latexsym}
\usepackage{xspace}
\usepackage{epsfig}
\usepackage{float}
\usepackage{color}
\usepackage{fancybox}
\usepackage{colordvi}
\usepackage{multicol}
\usepackage{colordvi}

\theoremstyle{theorem}
\newtheorem{theorem}{Theorem}[section]
\newtheorem{prop}[theorem]{Proposition}
\newtheorem{lemma}[theorem]{Lemma}
\newtheorem{coro}[theorem]{Corollary}

\newtheorem{prop-def}{Proposition-Definition}
\newtheorem{coro-def}{Corollary-Definition}[section]
\theoremstyle{definition}
\newtheorem{defn}[theorem]{Definition}
\newtheorem{remark}[theorem]{Remark}

\newcommand{\nc}{\newcommand}
\nc{\on}{\operatorname}
\nc{\tred}[1]{\textcolor{red}{#1}}
\nc{\tblue}[1]{\textcolor{blue}{#1}}
\nc{\tgreen}[1]{\textcolor{green}{#1}}
\nc{\tpurple}[1]{\textcolor{purple}{#1}}
\nc{\btred}[1]{\textcolor{red}{\bf #1}}
\nc{\btblue}[1]{\textcolor{blue}{\bf #1}}
\nc{\btgreen}[1]{\textcolor{green}{\bf #1}}
\nc{\btpurple}[1]{\textcolor{purple}{\bf #1}}

\newcommand{\efootnote}[1]{}

\renewcommand{\textbf}[1]{}

\newcommand{\delete}[1]{}
\nc{\dfootnote}[1]{{}}          
\nc{\ffootnote}[1]{\dfootnote{#1}}
\nc{\mfootnote}[1]{\footnote{#1}} 
\nc{\ofootnote}[1]{\footnote{\tiny Older version: #1}}

\delete{
\nc{\mlabel}[1]{\label{#1}}  
\nc{\mcite}[1]{\cite{#1}}  
\nc{\mref}[1]{\ref{#1}}  
\nc{\mcite}[1]{\cite{#1}{{\bf{{\ }(#1)}}}}  
}

\nc{\mlabel}[1]{\label{#1}  
{\hfill \hspace{1cm}{\bf{{\ }\hfill(#1)}}}}
\nc{\mcite}[1]{\cite{#1}{{\bf{{\ }(#1)}}}}  
\nc{\mref}[1]{\ref{#1}{{\bf{{\ }(#1)}}}}  

\nc{\mbibitem}[1]{\bibitem[\bf #1]{#1}} 
\nc{\mkeep}[1]{\marginpar{{\bf #1}}} 

\nc{\repmap}{representative~map~}

\nc{\opa}{\ast} \nc{\opb}{\odot} \nc{\op}{\bullet} \nc{\pa}{\frakL}
\nc{\arr}{\rightarrow} \nc{\lu}[1]{(#1)}
\nc{\opc}{\sharp}\nc{\opd}{\natural}
\nc{\ope}{\circ}
\nc{\bin}[2]{ (_{\stackrel{\scs{#1}}{\scs{#2}}})}  
\nc{\binc}[2]{ \left (\!\! \begin{array}{c} \scs{#1}\\
    \scs{#2} \end{array}\!\! \right )}  
\nc{\bincc}[2]{  \left ( {\scs{#1} \atop
    \vspace{-1cm}\scs{#2}} \right )}  
\nc{\bs}{\bar{S}} \nc{\cosum}{\sqsubset} \nc{\la}{\longrightarrow}
\nc{\rar}{\rightarrow} \nc{\dar}{\downarrow} \nc{\dprod}{**}
\nc{\dap}[1]{\downarrow \rlap{$\scriptstyle{#1}$}}
\nc{\md}{\mathrm{dth}} \nc{\uap}[1]{\uparrow
\rlap{$\scriptstyle{#1}$}} \nc{\defeq}{\stackrel{\rm def}{=}}
\nc{\disp}[1]{\displaystyle{#1}} \nc{\dotcup}{\
\displaystyle{\bigcup^\bullet}\ } \nc{\gzeta}{\bar{\zeta}}
\nc{\hcm}{\ \hat{,}\ } \nc{\hts}{\hat{\otimes}}
\nc{\barot}{{\otimes}} \nc{\free}[1]{\bar{#1}}
\nc{\uni}[1]{\tilde{#1}} \nc{\hcirc}{\hat{\circ}} \nc{\lleft}{[}
\nc{\lright}{]} \nc{\lc}{\lfloor} \nc{\rc}{\rfloor}
\nc{\curlyl}{\left \{ \begin{array}{c} {} \\ {} \end{array}
    \right .  \!\!\!\!\!\!\!}
\nc{\curlyr}{ \!\!\!\!\!\!\!
    \left . \begin{array}{c} {} \\ {} \end{array}
    \right \} }
\nc{\longmid}{\left | \begin{array}{c} {} \\ {} \end{array}
    \right . \!\!\!\!\!\!\!}
\nc{\onetree}{\bullet} \nc{\ora}[1]{\stackrel{#1}{\rar}}
\nc{\ola}[1]{\stackrel{#1}{\la}}
\nc{\ot}{\otimes} \nc{\mot}{{{\boxtimes\,}}}
\nc{\otm}{\overline{\boxtimes}} \nc{\sprod}{\bullet}
\nc{\scs}[1]{\scriptstyle{#1}} \nc{\mrm}[1]{{\rm #1}}
\nc{\margin}[1]{\marginpar{\rm #1}}   
\nc{\dirlim}{\displaystyle{\lim_{\longrightarrow}}\,}
\nc{\invlim}{\displaystyle{\lim_{\longleftarrow}}\,}
\nc{\mvp}{\vspace{0.3cm}} \nc{\tk}{^{(k)}} \nc{\tp}{^\prime}
\nc{\ttp}{^{\prime\prime}} \nc{\svp}{\vspace{2cm}}
\nc{\vp}{\vspace{8cm}} \nc{\proofbegin}{\noindent{\bf Proof: }}
\nc{\proofend}{$\blacksquare$ \vspace{0.3cm}}
\nc{\modg}[1]{\!<\!\!{#1}\!\!>}
\nc{\intg}[1]{F_C(#1)} \nc{\lmodg}{\!
<\!\!} \nc{\rmodg}{\!\!>\!}
\nc{\cpi}{\widehat{\Pi}}
\nc{\sha}{{\mbox{\cyr X}}}  
\nc{\shap}{{\mbox{\cyrs X}}} 
\nc{\shpr}{\diamond}    
\nc{\shp}{\ast} \nc{\shplus}{\shpr^+}
\nc{\shprc}{\shpr_c}    
\nc{\msh}{\ast} \nc{\zprod}{m_0} \nc{\oprod}{m_1}
\nc{\vep}{\varepsilon} \nc{\labs}{\mid\!} \nc{\rabs}{\!\mid}
\nc{\mmbox}[1]{\mbox{\ #1\ }}

\nc{\rsd}{\operatorname{RSD}}
\nc{\Ker}{\operatorname{Ker}}
\nc{\mult}{\operatorname{mult}}
\nc{\diff}{\mathfrak{Diff}}
\nc{\rank}{\operatorname{Rank}}
\nc{\stab}{\operatorname{Stab}}
\nc{\aut}{\operatorname{Aut}}
 \nc{\fp}{\operatorname{FP}}
\nc{\rchar}{\operatorname{char}}
\nc{\End}{\operatorname{End}}
\nc{\Fil}{\operatorname{Fil}}
\nc{\Mor}{\operatorname{Mor}\xspace}
 \nc{\gmzvs}{gMZV\xspace}
\nc{\gmzv}{gMZV\xspace}
 \nc{\mzv}{MZV\xspace}
\nc{\mzvs}{MZVs\xspace}
\nc{\Hom}{\operatorname{Hom}}
\nc{\id}{\operatorname{id}}
\nc{\im}{\operatorname{im}}
 \nc{\incl}{\operatorname{incl}}
 \nc{\map}{\operatorname{Map}}
\nc{\mchar}{\oparatorname{char}}
\nc{\nz}{\rm NZ}
 \nc{\supp}{\mathrm Supp}
 \nc{\Int}{\mathbf{Int}}
\nc{\Mon}{\mathbf{Mon}}

\nc{\Alg}{\mathbf{Alg}} \nc{\Bax}{\mathbf{Bax}} \nc{\bff}{\mathbf f}
\nc{\bfk}{{\bf k}} \nc{\bfone}{{\bf 1}} \nc{\bfx}{\mathbf x}
\nc{\bfy}{\mathbf y}
\nc{\base}[1]{\bfone^{\otimes ({#1}+1)}} 
\nc{\Cat}{\mathbf{Cat}}

\nc{\detail}{\marginpar{\bf More detail}
    \noindent{\bf Need more detail!}
    \svp}

\nc{\rbtm}{{shuffle }} \nc{\rbto}{{Rota-Baxter }}
\nc{\remarks}{\noindent{\bf Remarks: }} \nc{\Rings}{\mathbf{Rings}}
\nc{\Sets}{\mathbf{Sets}} \nc{\wtot}{\widetilde{\odot}}
\nc{\wast}{\widetilde{\ast}} \nc{\bodot}{\bar{\odot}}
\nc{\bast}{\bar{\ast}} \nc{\hodot}[1]{\odot^{#1}}
\nc{\hast}[1]{\ast^{#1}} \nc{\mal}{\mathcal{O}}
\nc{\tet}{\tilde{\ast}} \nc{\teot}{\tilde{\odot}}
\nc{\oex}{\overline{x}} \nc{\oey}{\overline{y}}
\nc{\oez}{\overline{z}} \nc{\oef}{\overline{f}}
\nc{\oea}{\overline{a}} \nc{\oeb}{\overline{b}}
\nc{\weast}[1]{\widetilde{\ast}^{#1}}
\nc{\weodot}[1]{\widetilde{\odot}^{#1}} \nc{\hstar}[1]{\star^{#1}}
\nc{\lae}{\langle} \nc{\rae}{\rangle}
\nc{\lf}{\lfloor}\nc{\rf}{\rfloor}

\nc{\cala}{{\mathcal A}} \nc{\calb}{{\mathcal B}}
\nc{\calc}{{\mathcal C}}
\nc{\cald}{{\mathcal D}} \nc{\cale}{{\mathcal E}}
\nc{\calf}{{\mathcal F}} \nc{\calg}{{\mathcal G}}
\nc{\calh}{{\mathcal H}} \nc{\cali}{{\mathcal I}}
\nc{\call}{{\mathcal L}} \nc{\calm}{{\mathcal M}}
\nc{\caln}{{\mathcal N}} \nc{\calo}{{\mathcal O}}
\nc{\calp}{{\mathcal P}} \nc{\calr}{{\mathcal R}}
\nc{\cals}{{\mathcal S}} \nc{\calt}{{\mathcal T}}
\nc{\calu}{{\mathcal U}} \nc{\calw}{{\mathcal W}} \nc{\calk}{{\mathcal K}}
\nc{\calx}{{\mathcal X}} \nc{\CA}{\mathcal{A}}

\nc{\fraka}{{\mathfrak a}} \nc{\frakA}{{\mathfrak A}}
\nc{\frakb}{{\mathfrak b}} \nc{\frakB}{{\mathfrak B}}
\nc{\frakD}{{\mathfrak D}} \nc{\frakg}{{\mathfrak g}}
\nc{\frakH}{{\mathfrak H}} \nc{\frakL}{{\mathfrak L}}
\nc{\frakM}{{\mathfrak M}} \nc{\bfrakM}{\overline{\frakM}}
\nc{\frakm}{{\mathfrak m}} \nc{\frakP}{{\mathfrak P}}
\nc{\frakN}{{\mathfrak N}} \nc{\frakp}{{\mathfrak p}}
\nc{\frakS}{{\mathfrak S}}

\font\cyr=wncyr10 \font\cyrs=wncyr7
\nc{\gl}[1]{\textcolor{red}{LG:#1}}
\nc{\ql}[1]{\textcolor{blue}{QL:#1}}
\nc{\red}[1]{\textcolor{red}{#1}}
\nc{\blue}[1]{\textcolor{blue}{#1}}

\allowdisplaybreaks[4]
\begin{document}

\title{Representations of Rota-Baxter Algebras and Regular-Singular Decompositions}

\author{Zongzhu Lin}
\address{Department of Mathematics,
Kansas State University,
Mahattan, Kansas 66506}
\email{zlin@math.ksu.edu}
\author{Li Qiao}
\address{Department of Mathematics, Landzhou University, Lanzhou, China}\email{liqiaomail@126.com}

%
\begin{abstract} There is a Rota-Baxter algebra structure on the field  $A=\bfk((t))$
with $ P$ being the projection map
$A=\bfk[[t]]\oplus t^{-1}\bfk[t^{-1}]$ onto
$ \bfk[[ t]]$. We study the representation theory  and regular-singular decompositions of any finite dimensional $A$-vector space.
The main result shows that the category of finite dimensional representations
is semisimple and consists of exactly three isomorphism classes of
irreducible representations which are all one-dimensional. As a consequence,
the number of $GL_A(V)$-orbits in the set of all regular-singular decompositions
of an $n$-dimensional $ A$-vector space $V$ is $(n+2)(n+1)/2$. We also use the result to compute the generalized class number, i.e., the number of the $GL_n(A)$-isomorphism classes of finitely generated $\bfk[[t]]$-submodules of $A^n$. 
\end{abstract}

\maketitle


\setcounter{section}{0}

\section{Introduction}
\subsection{}

A Rota-Baxter algebra of weight $\lambda$ consists of an associative algebra $A$, equipped with a linear operator $P: A \longrightarrow A$ satisfying the Rota-Baxter relation:
\begin{equation}\label{eq:rbrel}
P(x)P(y) = P(P(x)y)+P(xP(y))+\lambda P(xy).
\end{equation}
$P$ is called a Rota-Baxter operator of weight $\lambda$.

The study of Rota-Baxter algebras originated in the work of Baxter \cite{Ba} on fluctuation theory, and the algebraic study was started by Rota \cite{Rot1}. The theory of Rota-Baxter algebras develops a general framework of the algebraic and combinatorial structures underlying integral calculus which is like differential algebras for differential calculus. Rota-Baxter algebra also finds its applications in combinatorics, mathematics physics, operads and number
theory \cite{Ag3,A-M,And,Bai,BBGN,BGP,Ca,C-K,C-K1,E,E-G1,
E-G4,EGGV,E-Gs,G-Z}. See~\mcite{Guw,Gub} for further details.

A simple but important
example of a Rota-Baxter algebra is the algebra of Laurent series
$$\bfk((t)) =\left\{\sum_{n=k}^{\infty}a_{n}t^{n}~\big|~a_{n} \in \bfk, -\infty < k < \infty \right\}
$$
where the Rota-Baxter operator is the projection
$$
P(\sum_{n=k}^{\infty}a_{n}t^{n}) = \sum_{n \geq 0} a_{n}t^{n}.
$$
Elements of $ \bfk((t))$ is the completion of the local ring of  rational functions around $0$ with regular singularity. The operator $ P$ maps a function to its regular part. Thus a function $f$ can be decomposed into a sum $f=(f-P(f))+P(f)$.  The Rota-Baxter relations suggests that $ P$ is an integral operator of the function spaces with values in the space of regular functions. Such a decomposition has found its application  in the renormalization of quantum field theory. See  \cite{C-K,EGK3,EGM,C-M} for more details.

The concept of representations of Rota-Baxter algebras was introduced in \cite{G-Lin}. Similar to representations of the differential algebras, which can be thought as a vector bundle with a connection,  a representation of a Rota-Baxter algebra  is a module of the algebra together with an integral operator $P$. A representation  can also be interpreted in terms of Feyman integrals of vector bundles with a compatible integral operator, which decomposes a section into regular part and singular part.

\subsection{} In this paper, we concentrate on the representations of the
Rota-Baxter algebra $A=\bfk((t))$ with the Rota-Baxter operator
$ P: A\rightarrow A$ being the projection of $A$ to $ \calo=\bfk[[t]]$.
Through the decomposition $A=\bfk[[t]]\oplus t^{-1}\bfk[t^{-1}]  $
as $\bfk$-vector spaces, we view $ A$ is field of  formal functions around the point
$0$ with regular singularity at $0$, then this decomposition is meant to decompose
each function into regular part (those in $\bfk[[t]]$) and singular part
(those in $ t^{-1}\bfk[t^{-1}])$. Such decomposition is also
reflected in the representations of the Rota-Baxter algebra.  In fact for each
$(A, P)$-module
$ (V, p)$ with $ V$ being an $A$-vector space, the pair of
$\bfk$-subspaces $ (p(V), \ker(p))$ defines a decomposition
$ V=p(V)\oplus \ker(p)$ of the $\bfk$-vector space. Here
$ p(V)$ is a $ \bfk[[t]]$-submodule of $V$ and $ \ker(p)$ is a
$\bfk[t^{-1}]$-submodule of $V$. We will call such decomposition
of $V$ as a {\em regular-singular decomposition}  (RSD for short) of $V$.
We will prove that on any $A$-vector space $V$, an $(A, P)$-module structure
is equivalent to giving an RSD of $V$.  We will study the isomorphism classes
of all RSDs of $V$ under the group $GL_A(V)$-action and use this to classify
all $(A, P)$-modules which are finite dimensional over $A$.

The field $A$ is a complete field with a discrete evaluation and the ring
of integers $\calo=\bfk[[t]]$. In fact $\calo$ is a split complete discrete valuation in the sense that the residue field is a subring.   Each representation $ (V, p)$ of
dimension $n$ as an $A$-vector space defines an $\calo$-submodule
$ M=p(V)$, which can be thought as an element of the set $\calg_n(\calo)$ of
all $ \calo$-submodules.   The group $GL_n(A)$ acts on $\calg_n(\calo)$. We will also
study the $GL_n(A)$-orbits in $ \calg_n(\calo)$. The  map $\on{RB}(V)\rightarrow \calg_n(\calo)$ ($p\mapsto p(V)$) is $GL_n(A)$-equivariant. We will show that this map is fiber bundle with fiber $GL_n(\calo)$-orbits of $\bfk[t^-1]$-submodules of $V$.  It turns these structure are closely related the
affine Grassmannians $GL_n(A)/GL_n(\calo)$. The $GL_n(\calo)$ of $ \bfk[t^{-1}]$-lattices in $ A^n$ opposite affine Grassmannians.

\begin{theorem}   For the complete discrete valuation field $A=\bfk((t))$ with valuation
ring $\calo=\bfk((t))$, and any $n$, the number of $GL_n(A)$-orbits in
$\calg_{n}(\calo)$ is finite.
\end{theorem}

We remark that the ideal class group of a Dedekind domain can be arbitrarily large by a result of Claborn \cite{Cl, LG}.  We don't know if the theorem above holds with the complete discrete valuation ring $\calo$ replaced by any PID. One also notices that for any Dedekind domain $D$ with field $K$ of fractions. If $ \calm_n(D)$ is the set of all finitely generated $D$-submodules  of $K^n$. The group $GL_n(K)$ acts on $\calm_{n}(D)$. If we denote by $cl_n(D)=\calm_n(D)/GL_n(K)$ the set of all $GL_n(K)$-orbits in $\calm_n(D)$. Then $\calm_1(D)/GL_1(K)$ is the ideal class group $cl(D)$ of $D$ and it acts on the set $\calm_n(D)/GL_n(K)$, by $J\cdot M=JM$ for each fractional ideal $J$ in $K$ and each finitely generated $D$-submodule $M$ of $K^n$. One of the question is to describe the $cl(D)$-orbits in $ cl_n(D)$. In the case that $ D$ is complete discrete evaluation ring, there are only finitely many $ cl(D)$-orbits in $cl_n(D)$ for all $n$.

\subsection{} The main results are the following.

(a) The category of representations of finite dimensional $(A, P)$-modules is
a semisimple abelian category with exactly three  isomorphism classes of
irreducible objects. They are
$(A, 0)$, $(A, \on{Id})$, and $ (A, P)$ (Theorem~\ref{rb-decomp}).

(b) If $ V$ is an $n$-dimensional $A$-vector space, then the number of
$GL_A(V)$-orbits on the set of all $\rsd$s of $V$ is $(n+2)(n+1)/2$, which indexed by the triples $(k,r, l)\in \mathbb{Z}^{3}_{\geq 0}$ with $ k+r+l=n$ (Corollary 4.6).

(c) The number of $ GL_n(A)$-orbits in $\calg_n(\calo)$ is finite and in one-to-one correspondence to the isomorphism classes of $ n$-dimensional RB-representations (Corollary 5.2).

(d) The $ GL_n(A)$-orbit of type $(k,r,l)\in \mathbb{Z}^{3}_{\geq 0}$ in $\calg_n(\calo)$ is a fibre bundle over the $A$-variety of all flags of type $(k, r, l)$ with fibre being the affine Grassmanian $GL_r(A)/GL_r(\calo)$.

(e) The $GL_n(A)$-orbits in $ \rsd(A^n)$ is also characterised as quotient groups  (Theorem 5.6).

We remark that the set of all RSDs in $V$ is closely related to affine Grassmanians and has
an infinite dimensional variety structure. The geometric structure of these orbits
will be studied elsewhere. Also the argument in the paper works for any complete discrete valuation ring in its field of fractions. The applications of vector bundles over curves will be studied in a future paper.

\subsection{} The main tool in classifying all representations  on a finite dimensional
$A$-vector space is to first classify all $ \bfk[[t]]$-submodules $M$ of $V$.
In case $V$ is one-dimensional, then all such $M$ are straightforward.
For higher dimensions, each such $M$ contains a largest
$\bfk[[t]]$-submodule $\widetilde{M}$ consisting of all divisible
elements. $\widetilde{M}$ is the largest $ A$-subspace of $V$ contained in $M$.
One of the results, which might be well known to number theorists, is that $M$
is a finitely generated $\calo$-module if $ \tilde{M}=0$. But we could not find a good
reference and provide a proof here. In particular $M/\widetilde{M}$ is a finitely
generated $\bfk[[t]]$-submodule of $ V/\widetilde{M}$. Using this results,
we are able to decompose every $(A,P)$-module as direct sum of three types of
modules, one part is $p=0$, which will be called {\em trivial}, one part is $p=\on{Id}$,
and the part satisfies $p(V)$ is lattice in $ V$, which we will called it {\em regular}.
This was done in Section 3. In Section 4, we classify all regular modules.
It turns out that two regular modules are isomorphic if and only if they have
the same dimension over $A$.

\subsection{Acknowledgement} This research was initiated when both authors
were participating the KITPC program, {\em Mathematical Method from Physics}
during the summer 2013 in Beijing. This program was supported by both KITPC
and Morningside Center of Mathematics. Authors thank both institutions for the
hosting the programs, in particular the authors appreciate the housing support
provided in MCM. The authors also thank Li Guo for stimulating discussions
and for encouragement.

\section{Regular-Singular decompositions of a module}
In this section, we first define the notion of a regular-singular decomposition (RSD) of vector spaces for the Laurent series Rota-Baxter algebra. Then we
give the definition of the Rota-Baxter module (RB-module). Finally, we identify the RSDs with RB-modules.


Let $\calo:=\bfk[[t]]$, then $\calo $ is a complete discrete valuation ring  with the maximal idea $t\calo$ and $\bfk$ the field of residue.  Its  group of units is
$\calo^{*}=GL_1(\calo)=\calo\backslash t\calo$. Let $A$ be the field of fractions of $\calo$, then $A=\bfk((t))$ be  the Laurent series algebra and $A^{-}:=t^{-}\bfk[t^{-1}] \subseteq \bfk[t^{-1}]$.   Then $A = \calo \oplus A^{-}$ as $\bfk$-vector space. $A$ has a Rota-Baxter algebra structure with Rota-Baxter operator $P: A  \rightarrow A$ which is the projection  onto $\calo$ under this direct sum decomposition.  Although we are only interested in the  field $\bfk$ for Rota-Baxter algebras. 

\begin{defn}
{\rm
Let $V$ be an $A$-vector space. A pair $(M,N)$ of $\bfk$-subspaces of $V$ is called  a {\em regular-singular decomposition} ({\bf RSD}) if
\begin{enumerate}
\item $V = M \oplus N$ as $\bfk$-vector space;
\item $M$ is an $\calo$-submodule of $V$;
\item $N$ is a $\bfk[t^{-1}]$-submodule of $V$.
\end{enumerate}
}
\end{defn}

Let
$$
\rsd(V):=\left\{(M,N)~\big|~(M,N)~\mbox{is a regular-singular decomposition of $V$}\right\}.
$$
Let $GL_{A}(V)$ be the group of invertible $A$-linear transformations on $V$. Then $GL_{A}(V)$ acts on $\rsd(V)$ by
$$
\phi(M,N) = (\phi(M),\phi(N)), \quad \forall \phi \in GL_{A}(V).
$$
Note that $\phi$ is $A$-linear, so $\phi(M)$ is an $\calo$-submodule of $V$ and $\phi(N)$ is a $\bfk[t^{-1}]$-submodule.

Next we want to determine the $GL_{A}(V)$-orbits in $\rsd(V)$.

\begin{defn}

An $A$-module $V$ with a $\bfk$-linear map $p:V\longrightarrow V$ is called a {\em $Rota$-$Baxter$ ($RB$)-module over $(A,P)$} if
\begin{equation}\label{eq:rbmodule}
P(a)p(x) = p(P(a)x + ap(x) - ax), \quad \forall a \in A, x \in V.
\end{equation}
Such $p$ is called a {\em $Rota$-$Baxter$ representation operator on $V$}, the pair $(V,p)$ is called an {\em $RB$-module over $(A,P)$}.
\end{defn}

Let
$$
RB(V):=\left\{p:V\longrightarrow V~\big|~p ~\mbox{is a Rota-Baxter representation operator on $V$} \right\}.
$$
An RB-module homomorphism $f: (V,p) \longrightarrow (W,p')$ between two RB-modules  $(V,p)$ and $(W,p')$  is an $A$-module homomorphism $f: V \longrightarrow W$ such that $f \circ p = p' \circ f$. $(V,p)$ and $(W,p')$ are isomorphism if $f$ is an $A$-module isomorphism. In this case $V \cong W$ and $p' = f \circ p \circ f^{-1}$.

Then for an $RB$-algebra $(A,P)$, the category of all $RB$-modules over $(A,P)$  which are finitely generated as $A$-modules is an abelian category.  We are interested in classifying all isomorphism classes of this category. For each fixed finitely generated $A$-module $V$, the group $GL_{A}(V)$ acts on the set $RB(V)$ by conjugation $\phi\cdot p=\phi\circ p\circ \phi^{-1}$ for any $\phi\in GL_{A}(V)$, $p\in RB(V)$.

 Our goal is to classify the isomorphism classes of $RB$-module structures  on $V$. This is amount to classifying the $GL_{A}(V)$-orbits in $RB(V)$.

\begin{prop}\label{prop:idem}
For an $A$-module $V$, a $\bfk$-linear map $p:V\longrightarrow V$ is a $Rota$-$Baxter$ representation operator  if and only if $p$  satisfies the following conditions:
\begin{enumerate}
\item[{\rm(a)}] $p^{2} =p$;
\item[{\rm (b)}] $p(V)$ is an $\calo$-submodule of $V$;
\item[{\rm (c)}] $\Ker(p)$ is a $\bfk[t^{-1}]$-submodule of $V$.
\end{enumerate}
\end{prop}

\begin{proof}
It follows from Eq. (\ref{eq:rbmodule}) that $p$ is a $Rota$-$Baxter$ representation operator if and only if
$$
\left\{\begin{array}{ll} {\rm (i)} \quad ap(x) = p(ap(x)),&\forall a\in \calo, x \in V;\\
{\rm (ii)} \quad  p(ax) = p(ap(x)),& \forall a \in \bfk[t^{-1}], x \in V. \end{array}\right.
$$

Taking $a=1$ in (i) or (ii), we have $p^{2}=p$. Furthermore, (i) implies that $p(V)$ is an $\calo$-submodule of $V$ and (ii) implies $\ker(p)$ is a $\bfk[t^{-1}]$-submodule of $V$.

Conversely, assume that $p$ satisfies (a), (b) and (c). Then for any $a\in \calo$, $x \in V$,  the fact that $
ap(x) \in p(V)  $ implies that $ ap(x) = p(y), $ for some $y \in V$.
By (a), we have $p(y) = p(p(y))$. Thus (i) holds.

Similarly $p^{2}=p$, implies that $x-p(x) \in \ker (p)$ for all $x\in V$. By (c), $a(x-p(x)) \in \ker (p)$ for any $a\in \bfk[t^{-1}]$. Thus $p(a(x-p(x))) =0$ and  $p(ax)=p(ap(x))$. Hence  (ii) holds.
\end{proof}

\begin{theorem}\label{theorem p} For any $A$-module $V$,
the map $\Psi:RB(V) \longrightarrow \rsd(V)$ defined by $\Psi(p) = (\im (p),\ker (p))$ is a $GL_{A}(V)$-equivalent bijection.
\end{theorem}

\begin{proof}
By Proposition \ref{prop:idem}, $\im(p)=p(V)$ is an $\calo$-submodule of $V$ and $\ker(p)$ is a $\bfk[t^{-1}]$-submodule of $V$. Since $p^{2}=p$, then $V = \im (p) \oplus \ker(p)$ as $\bfk$-vector space. Hence, $\Psi(p) = (\im (p),\ker(p))$ is a regular-singular decomposition of $V$.

It is straightforward to verify that
$$
\Psi(\phi \circ p \circ \phi^{-1}) = (\phi(\im(p)),\phi(\Ker(p))), \quad \forall \phi\in GL_{A}(V).
$$
To define the inverse map $\Psi^{-1}: \rsd(V) \longrightarrow RB(V)$, for each $(M,N) \in \rsd(V)$, one defines $p : V \longrightarrow V$ by $p=p_{r}|_{M}$ the projection of $V$ to $M$. Then $p$ satisfies (a), (b) and (c) of Proposition \ref{prop:idem}. Hence for any $(M, N)\in \rsd(V)$, $p\in RB(V)$, $\Psi\circ \Psi^{-1}(M, N)=(M, N)$, $\Psi^{-1}\circ \Psi(p)=(p)$.\end{proof}

We will use Theorem \ref{theorem p} to classify the isomorphism classes of $RB$-module structure on an $A$-module $V$.

Let us first work on special case $ V=A$ and classify regular-singular decompositions of $V$ up to isomorphisms. It is easy to see that if $M=A$ and $N=0$, then $p=id$; if $M=0$, $N=A$, then $p=0$.

\begin{lemma}\label{lemma2.3}
Let $M\subseteq A$ be an $\calo$-submodule. If $\displaystyle x=\sum_{i=n}^{\infty}c_{i}t^{i}\in M$ where $c_{n}\neq 0$, then $t^{n}\in M$, $\calo t^{n}\subseteq M$.
\end{lemma}
\begin{proof}
Let $\displaystyle x=\sum_{i=n}^{\infty}c_{i}t^{i}=t^{n}\sum_{i=0}^{\infty}c_{i+n}t^{i}\in M$ and $\displaystyle b:=\sum_{i=0}^{\infty}c_{i+n}t^{i}$ where $c_{n}\neq 0$. Then $b\in \calo^{*}$, that is $b$ is an invertible element in $\calo$. Thus we have $b^{-1}\in \calo$, $t^{n}=b^{-1}x\in M$ and $\calo t^{n}\subseteq M$.
\end{proof}

Let $t\in \calo$ be the generator of the maximal idea of $\calo $, and for any $x\in A$, define
$$\nu(x):= min~\{ n \in \mathbb{Z}~|~x\in~t^n\calo~\},$$
then $\nu(x)$ is the valuation of $x$ at the unique maximal ideal $(t)\subseteq \calo$

\begin{prop} \label{prop:2.4} For $ V=A$ as $A$-module,
if $(M, N)\in \rsd(A)$ and $M\neq 0$, $M\neq A$, then there exists $n\in \mathbb{Z}$ such that $M=\calo t^{n}$.
\end{prop}

\begin{proof}
If $M\neq 0$ and $M\neq A$, then $N\neq 0$, $N\neq A$. Thus there exists $0\neq z\in N$ with $\displaystyle z=\sum_{i=m}^{\infty}c_{i}t^{i}(c_{m}\neq 0)$. And for any $0\neq x\in M$, let $\displaystyle x=\sum_{j=n}^{\infty}a_{j}t^{j}(a_{n}\neq 0)$, then $n>m$. Otherwise, if there exists $x\in M$ such that $\nu(x)\leq m$. Then $t^{\nu(x)}\calo\subseteq M$, $z\in t^{\nu(x)}\calo\subseteq M$ gives a contradiction.

For any $x\in M\backslash \{0\}$, $z\in N\backslash \{0\}$, choose $n$ to be minimal integer such that $n\geq \nu(z)$. Then we have $M=\calo t^{n}$ by the Lemma \ref{lemma2.3}.
\end{proof}

\begin{remark} Proposition \ref{prop:2.4} shows that for the complete discrete valuation
ring $ \calo$ with field of fractions $A$, an $\calo$-submodule $M$ of $ A$ is a
fractional ideal if and only if $ M\neq A$ and $M\neq 0$. Since the ideal class group of a discrete
valuation ring is always identity, thus there is an element $ 0\neq a \in A$ such that
$ aM=\calo$. The completeness is essential. This does not generalise to non-local PIDs as one can easily find
$\bfk[t]$ submodules of $ \bfk(t)$ that are not fractional ideals.
\end{remark}
\begin{coro}  \label{coro:2.7}
For $ V=A$ as $A$-module, every $(M, N)\in \rsd(A)$  is isomorphic to $(\calo, N')$.
\end{coro}

\begin{proof}
In fact, every RSD of $A$, which is neither divisible nor trivial, has the form $(\calo t^{n}, N)$ for some $n$. Then apply the invertible $\phi \in GL_{A}(V)$ defined by
 $\phi(x)=t^{n}x$ for any $x\in V$, and some n.
\end{proof}

\begin{lemma}
Let $(\calo, N)$ be a regular-singular decomposition of A. Then there exist $b_{n}=t^{-n}+r_{n}(t)$ with $r_{n}(t)\in \calo$ for any $n\geq 1$ such that $\{b_{n}~|~n\geq 1\}$ is a $\bfk$-basis of N.
\end{lemma}

\begin{proof}
Since $A=\calo \oplus N$, then the restriction of the quotient map $A\rightarrow A/\calo$ to $N$ defines a $\bfk$-linear isomorphism $ A/\calo \cong N$. Let $b_n$ be the inverse image of $t^{-n}$ for each $n\geq 1$. Then $b_n$ has the required form.
\end{proof}

\begin{lemma} For any $n\geq 1$,
 $b_{n}\in \bfk[t^{-1}]b_{1}$,
\end{lemma}

\begin{proof}
We use induction on $n\geq 1$ to prove the statement. If $n=1$, it is clear that $b_{1}\in \bfk[t^{-1}]b_{1}$. For $n\geq 2$, assume that $b_{n}\in \bfk[t^{-1}]b_{1}$ . Then
\begin{eqnarray*}
b_{n+1}-t^{-1}b_{n}+r_{n}(0)b_{1}&=&t^{-n-1}+r_{n+1}(t)-t^{-n-1}-t^{-1}
r_{n}(t)+r_{n}(0)t^{-1}+r_{n}(0)r_{1}(t)\\
&=&-(r_{n}(t)-r_{n}(0))t^{-1}+r_{n+1}(t)+r_{n}(0)r_{1}(t)\in \calo.
\end{eqnarray*}

Since $b_{n+1}-t^{-1}b_{n}+r_{n}(0)b_{1}\in N$, we have $b_{n+1}-t^{-1}b_{n}+r_{n}(0)b_{1}=0$.  Then by induction assumption
$b_{n+1}=t^{-1}b_{n}-r_{n}(0)b_{1}\in \bfk[t^{-1}]b_{1}.$
\end{proof}

\begin{prop} \label{prop:2.10}
Let $(\calo, N)\in \rsd(A)$. Then there exists $b\in \calo^{\ast}$ such that $N=A^{-}b$.
\end{prop}

\begin{proof}
We have $N\subseteq \bfk[t^{-1}]b_{1} \subseteq N$.
Write $b_{1}=t^{-1}b$ where $b=1+t\cdot r_{1}(t)\in \calo^{*}$, then $N=A^{-}b$.
\end{proof}

\begin{theorem} Let $(M, N) \in \rsd(A)$, then either $(M,N)=(A, 0)$, or $(M,N)=(0,A)$, or $(M,N)\cong (\calo, A^-)$.
\end{theorem}
\begin{proof} By Corollary~\ref{coro:2.7}, we can assume $M=\calo$. Now by Proposition~\ref{prop:2.10}, we can assume $ N=A^{-1}b$ for some $b\in \calo^*$. The map $ \phi: A\rightarrow A$ defined by $ \phi(x)=xb^{-1}$ satisfies $\phi(\calo)=\calo$. Hence, $ \phi(\calo, N)=(\calo, A^{-})$.
\end{proof}

\begin{coro} There are exactly three isomorphism classes of Rota-Baxter $(A,P)$-modules on $V=A$. They are $p=0$, $p=\on{Id}$, and $p=P$.
\end{coro}

\section{Finite generation and regular-singular decomposition in higher ranks}
Let us recall that for an $ \calo$-module $M$, an element $ x\in M$  is  called {\em divisible} if for  any $0\neq a\in \calo$, there is a $y\in M$ such that $ x=ay$.  We call $M$ {\em divisible} if every element is divisible.

If  $M=A$, then  $M$ is  divisible. We will say that the corresponding $RB$-module is {\em divisible} . When $M=0$, we will say the corresponding $RB$-module is {\em trivial} since $p=0$. An $RB$-module $(V, p)$ is called {\em regular} is it does not have nonzero trivial and nonzero divisible submodules.

Before we proceed to classify all $RB$-modules over $A$, we need to prove a result that every $ \calo$-submodule $M$ of a finite dimensional  $A$-vector space $V$  is finitely generated if and only if $M$ has no divisible vector.
The result in this section works for any complete discrete valuation ring $ \calo$ with $A$ being the field of fractions. The results in this section should be well-known to number theorists on fractional submodules in a finite dimensional vector space, but we could not find a good reference and thus provide a proof here for the readers' convenience.  This result is definitely false for $\mathbb{Z}$  by considering the $\mathbb{Z}[q^{-1}]$-submodules of $V$ for any prime number $q$.

\subsection{}
Let $V$ be a finite dimensional $A$-vector space and $M\subseteq V$ be an $\calo$-submodule. Let
$$\widetilde{M}=\{x\in M~|~Ax\in M\}.$$
Then $\widetilde{M}$ is the set of all divisible elements in $M$ and  $\widetilde{M}\subseteq V$ is an $A$-vector subspace.

It is clear that if $M$ is a finitely generated $\calo$-submodule of $V$, then $\widetilde{M}=\{0\}$ since $M$ is torsion-free and thus a finitely generated free submodule of $V$.

Assume that $M$ is an $ \calo$-submodule of an $A$-vector space $V$. For any $0\neq x \in V$, let
$$I_M(x):=\{a\in A~|~ax\in M\}.$$
Then $I_M(x)$ is an $\calo$-submodule of $A$. By Proposition~\ref{prop:2.4},  exactly one of the following three cases appears: (a) $I_M(x)=A$ if and only if $ Ax\subseteq M$,  (b) $I_M(x)=\{0\}$ if and only if $ Ax\cap M=\{0\}$, and (c) $I_M(x)=\calo t^n$ for some $n \in \mathbb{Z}$ if and only if $  \{0\} \subsetneq Ax \cap M\subsetneq Ax $.

\begin{defn}  We call a regular-singular decomposition $(M, N)$ of an A-vector space V {\em regular} if  $\widetilde{M}=\{0\}$ and $ AM=V$.  We say that an $RB$-module over $(A, P)$ is {\em regular}, if the corresponding regular-singular decomposition is regular.
\end{defn} We will prove in Theorem~\ref{finite-generation} that $ \widetilde{M}=0$ implies that $ M$ is a finitely generated $\calo$-module and thus $ M$ is $\calo$-free of finite rank. Thus $(M, N)$ is regular if and only if $ M$ is a lattice in $V$. We will classify all regular RSDs in $V$ in Section \ref{sec:classify}.

We remark that the regularity condition corresponding to regular $D$-modules over a curve (see \cite[5.1.3]{HTT}).  Also $(M,N)$ is regular if and only if the corresponding $RB$-module does not have nonzero divisible and nonzero trivial $RB$-submodules, as we have defined earlier. 

For any $x\in V\backslash \{0\}$, define
$$\nu_M(x):=\min \{n\in \mathbb{Z} ~|~t^{n}x\in M\}.$$
It should be noted that it is possible that $ \nu_M(x)=-\infty $ or $+\infty$.

\begin{prop} \label{prop:3.2}
Let $V$ be an $A$-vector space and $M\subseteq V$ an $\calo$-submodule. For any $0\neq x\in V$,
then either there exists $n\in \mathbb{Z}$ such that $I_M(x)=\calo t^{n}$ with $n=\nu_M(x)$ or $I_M(x)=A$ with $x\in \widetilde{M} $ and $\nu_M(x)=-\infty$, or $I_M(x)=\{0\}$ with $ \nu_M(x)=+\infty$.
\end{prop}
\begin{proof}
If $\widetilde{M}$ is nonempty, then $0\neq x\in \widetilde{M}$ if and only if $I_M(x)=A$. In this case $ \nu_M(x)=-\infty$. If $ t^nx\not \in M$ for all $ n \in \mathbb{Z}$, then $ I_M(x)=\{0\}$ and $Ax\cap M=\{0\}$.

For $0\neq x\in V\backslash \widetilde{M}$ with $ \{0\}  \subsetneq Ax \cap M\subsetneq Ax $, there exists $n\in \mathbb{Z}$ such that $t^{n}x \in M$ and $t^{n-1}x \not\in M$. Since $M$ is an $\calo$-submodule, then we have $\calo t^{n}x \subseteq M$, i.e. $\calo t^{n} \subseteq I_M(x)$ and $\calo t^{n-1}\not\in I_M(x)$. Since $ \calo$ is a DVR, and all fractional ideals of $\calo $ in $A$ are of the form $ \calo t^m$ with $m \in \mathbb{Z}$ (see the proof of Proposition~\ref{prop:2.4}). Hence  $I_M(x)=\calo t^n$ with $ n=\nu_M(x)$.
\end{proof}

\begin{lemma}\label{lem:O-subm}
Let $V$ be a finite dimensional $A$-vector space and $M\subseteq V$ an $\calo$-submodule. If $0\neq z\in M$, set $\overline{V}:=V / Az$. For the quotient map $\pi:V\rightarrow \overline{V}$, we define  $\overline{M}:=\pi(M)\subseteq \overline{V}$ as an $\calo$-submodule of $\overline{V}$. If $\widetilde{M}=\{0\}$, then $\widetilde{\overline{M}}=\{0\}$.
\end{lemma}
\begin{proof}
If $\widetilde{M}=\{0\}$, then, by Proposition~\ref{prop:3.2}, there exists $k\in \mathbb{Z}$ such that $I_{M}(z)=\calo t^{k}$.  Now replacing $z$ with $t^{k}z$ and then we can assume that $ z\in M$ and $t^{-1}z \notin M$. So we have for any $a\in A$, $az\in M$ if and only if $a\in \calo$, i.e., $ I_M(z)=\calo$.

Suppose, on the contrary, that $\widetilde{\overline{M}}\neq \{0\}$. Then there exists $y\in M$ such that $0\neq \overline{y}=\pi(y) \in \overline{M}$ and $I_{\overline{M}}(\overline{y})=A$.

Since $\widetilde{M}=\{0\}$, then there exists $m \in \mathbb{Z}$ such that $I(y)=\calo t^{m}$. Similar to the case of $z$, replacing $y$ with $t^{m}y$, we can take $ y \in M $ such that
\begin{equation}\label{y-condition}
I_M(y)=\calo\; \text{ i.e., } \;\nu_M(y)=0, \; \text{ and }\; I_{\overline{M}}(\overline{y})=A.
\end{equation}
 Note that $y$ and $z$ are $A$-linearly independent in $V$.

By assumption, $t^{-n}y+Az \in \overline{M}$ for each $n\in \mathbb{N}$. Thus  there exists $a_{n}\in A$ such that $x_{n}=t^{-n}y-a_{n}z\in M$. Write $a_{n}=t^{\nu(a_{n})} r_{n}$ with $r_{n} \in \calo$. Since $t^{-1}y \notin M$ and $z \in M$, then $\nu(a_{n})<0$ when $n > 0$.

Since $t^{-n}y=x_{n}+a_{n}z$, we have $\nu(a_{n})=-n$. In fact, by
considering $t^{-n-\nu(a_{n})}y=t^{-\nu(a_{n})}x_{n}+r_{n}z \in M$,
we have $-n-\nu(a_{n})\geq 0$ and then $\nu(a_{n})\leq -n$.
Similarly, $t^{\nu(a_{n})+n} r_{n}z=y-t^{n}x_{n}\in M$ implies $\nu(a_{n})\geq -n$.
Thus $a_{n}=t^{-n} r_{n}$ with $r_{n}\in \calo^*$ and $x_{n}=t^{-n}(y-r_{n}z) \in M$.

For $n=1$, we have $r_{1}\neq 0$ and  $y-r_{1}z=tx_{1}\in tM$.
If we write $r_{1}:=c_{1}+td$ with $c_{1}\in \bfk$ and $d\in \calo$,
then $y-c_{1}z \in tM$.

Denote $y_{1}:=y$ and then $y_{1}-c_{1}z\in tM$. Denote $ y_2:=t^{-n_1}(y_1-c_1z) \in M$ for some  $n_{1}\geq 1$ such
that
$t^{-1}y_{2} \notin M$. Then
$\overline{y_{2}}=t^{-n_{1}}\overline{y_{1}}\neq 0$ and $\overline{y_{2}}$
still satisfies the same condition \eqref{y-condition} as  $y_{1}$ does. Hence there exists
$c_{2}\in \bfk$ such that $y_{2}-c_{2}z\in tM$.

Inductively, for each $\ell \in \mathbb{N}_{+}$, we denote
$y_{\ell}:=t^{-n_{\ell-1}}(y_{\ell-1}-c_{\ell-1}z) \in M$ with $ n_{\ell-1}\geq 1$ such that $t^{-1}y_{\ell} \notin M$ and there exists
$c_{\ell}\in \bfk$ such that $y_{\ell}-c_{\ell}z \in tM$.

Since $y_{\ell}=t^{-n_{\ell-1}}(y_{\ell-1}-c_{\ell-1}z) \in M$, we have $y_{\ell-1}=t^{n_{\ell-1}}y_{\ell}+c_{\ell-1}z$.
\begin{eqnarray*}
\displaystyle
y&=&y_{1}=t^{n_{1}}y_{2}+c_{1}z\\
&=&t^{n_{1}}(t^{n_{2}}y_{3}+c_{2}z)+c_{1}z\\
&=&t^{n_{1}+n_{2}}y_{3}+(c_{2}t^{n_{1}}+c_{1})z\\
&=&\cdots\\
&=&t^{\sum _{i=1}^{\ell}n_{i}}y_{\ell+1}+(\sum _{i=1}^{\ell}c_{i}t^{\sum _{j=1}^{i-1}n_{j}})z.
\end{eqnarray*}

Thus $\displaystyle y-(\sum _{i=1}^{\ell}c_{i}t^{\sum _{j=1}^{i-1}n_{j}})z
=t^{\sum _{i=1}^{\ell}n_{i}}y_{\ell+1}\in t^{\sum _{i=1}^{\ell}n_{i}}M$
for all $\ell \in \mathbb{N}_{+}$.
Since for any $n\in \mathbb{N}$, there exists $\ell \in \mathbb{N}_{+}$ such that $\displaystyle \sum _{i=1}^{\ell}n_{i}\geq n$, we have
$$
\displaystyle y-(\sum _{i=1}^{\infty}c_{i}t^{\sum _{j=1}^{i-1}n_{j}})z=y-(\sum _{i=1}^{\ell}c_{i}t^{\sum _{j=1}^{i-1}n_{j}})z-(\sum _{i=1}^{\infty}c_{\ell+i}t^{\sum _{j=1}^{i-1}n_{\ell+j}+\sum _{j=1}^{\ell}n_{j}})z \in \bigcap_{n=0}^{\infty}t^{n}M.
$$
Note that $\displaystyle \cap_{n=0}^{\infty}t^{n}M=\{0\}$ since $x\in \widetilde{M}$ if and only if $x \in \displaystyle \cap_{n=0}^{\infty}t^{n}M$.

Hence $\displaystyle y=(\sum _{i=1}^{\infty}c_{i}t^{\sum _{j=1}^{i-1}n_{j}})z$, but $\displaystyle \sum _{i=1}^{\infty}c_{i}t^{\sum _{j=1}^{i-1}n_{j}} \in \calo$. This implies $y\in Az$ and $\overline{y}=\{0\}$ contradicting the assumption that $\overline{y}\neq 0$. Hence $\widetilde{\overline{M}}=\{0\}$.
\end{proof}

\begin{remark} In the proof we only used the property  that $\calo$ is a complete discrete valuation ring. The fact that the residue field $ \bfk$ is a subring of $\calo$ is not essential as one can choose $ c_i\in \calo$ invertible.
\end{remark}

\begin{theorem} \label{finite-generation}
Let $V$ be a finite dimensional $A$-vector space and $M\subseteq V$ an $\calo$-submodule. If $\widetilde{M}=\{0\}$, then $M$ is a finitely generated $\calo$-submodule of V.
\end{theorem}
\begin{proof}
We use induction on $\dim_{A}V$. If $\dim_{A}V=1$, take any basis $z\in M$, then $I_{M}(z)=\calo t^{k}$ for some
$k\in \mathbb{Z}$, thus for any $a\in A$, $a\cdot t^{k}z \in M$ if and only if
$a\in \calo$, i.e. $M=\calo \cdot t^{k}z$.

Assume that the theorem is true for any vector spaces of $\dim V < n$, we
now consider $\dim V=n \geq 2$.

Take $0\neq z\in M$, and $\overline{V}:=V/Az$, for the quotient map
$\pi:V\longrightarrow \overline{V}$, we define
$\overline{M}=\pi(M)\subseteq \overline{V}$ be an $\calo$-submodule.
By Lemma \ref{lem:O-subm}$, \widetilde{\overline{M}}=\{0\}$. By the induction
assumption, $\overline{M}$ is a finitely generated $\calo$-submodule of
$\overline{V}$. Then we have a short exact sequence of $\calo$-modules:
$$0 \rightarrow M \cap Az \rightarrow M \rightarrow \overline{M} \rightarrow 0.$$
Because of $Az$ is 1-dimensional, $M \cap Az$ is an $\calo$-submodule
of $Az$, and $\widetilde{M \cap Az}=\{0\}$, then $M \cap Az$ is generated by
one element. So $M$ can be  generated by  finitely many elements as an $ \calo$-module.
\end{proof}

\begin{coro}\label{cor:fg}
Let $V$ be a finite dimensional $A$-vector space.
If $ M$ is an $ \calo$-submodule of $V$ such that $ \widetilde{M}=\{0\}$, then $M$ is a free $\calo$-module of rank $\dim_A AM$.
\end{coro}
\begin{proof} $ \calo$ is a DVR, a finitely generated module is free if and only if it is torsion free. 
\end{proof}

\begin{coro}\label{cor:freed}
Let $V$ be a finite dimensional $A$-vector space.
If $M$ is an $\calo$-submodule of $V$, then there is an $A$-linear  decomposition $V=\widetilde{M}\oplus V_0$ such that
$M=\widetilde{M}\oplus M_{f}$ with $M_f=M\cap V_0$ being a free
$\calo$-module of rank $ \dim_A AM-\dim_A \widetilde{M}$.
\end{coro}
\begin{proof} Define $\overline{V}=V/\widetilde{M}$ with the quotient map $ \pi: V \rightarrow \overline{V}$
which is an $ A$-linear map. Then $\overline{M}=M/\widetilde{M}$ is an
$ \calo$-submodule of $ \overline{V}$. We now show that $ \widetilde{\overline{M}}=\{0\}$.
Indeed, $\widetilde{\overline{M}}$ is an $A$-subspace of $ \overline{V}$ then
$ \pi^{-1}(\widetilde{\overline{M}})$ is an $ A$-subspace of $V$.
Since $\widetilde{M}\subseteq M$,
we have $ \pi^{-1}(\widetilde{\overline{M}})\subseteq M$.
Hence $ \pi^{-1}(\widetilde{\overline{M}})\subseteq \widetilde{M}$.
Thus $\widetilde{\overline{M}}=\{0\}$.
Now by Corollary \ref{cor:fg}, $\overline{M}$ is a free $ \calo$-module and the exact sequence
\[ 0\rightarrow \widetilde{M}\rightarrow M \rightarrow \overline{M}\rightarrow 0\]
of $ \calo$-modules splits. Let $ M_f$ be the image of a splitting map
(of $\calo$-modules). Thus we have $ M=\widetilde{M}\oplus M_f$ with
$ M_f$ being free of rank $\dim_A A\overline{M}=\dim_A A\overline{M}-\dim_A \tilde{M}$. Now decompose $V=\tilde{M}\oplus V_0$ with $ V_0$ be an $ A$-subspace of $V$ containing $M_f$. 
\end{proof}
Similar to the definition of $ \widetilde{M}$, for a $ \bfk[t^{-1}]$-submodule $ N$, we define $ \widetilde{N}=\{ x \in N\; |\; Ax \subseteq N\}$. Then $\widetilde{N}$ is an $A$-subspace of $V$. In general, $N/\widetilde{N}$ is not a finitely generated $\bfk[t^{-1}]$-submodule. However, the following theorem implies that when $(M, N)$ is in $\rsd(V)$, then $ N/\widetilde{N}$ is a finitely generated (and thus a free) $\bfk[t^{-1}]$-module.
\begin{theorem} \label{rsd-decomp}
Let $V$ be a finite dimensional $A$-vector space. If $ (M,N) \in \rsd(V)$, then there is a decomposition of
$ V=\widetilde{M}\oplus V_0\oplus \widetilde{N}$ as $ A$-vector space and
there is $(M_f, N_f)$  in $\rsd(V_0)$ such that $M=\widetilde{M}\oplus M_f$ as $ \calo$-module and
$ N=\widetilde{N}\oplus N_f$ as $ \bfk[t^{-1}]$-module, $M_f$ is a free $\calo$-module of finite rank and $ N_f$ is a free $\bfk[t^{-1}]$-module of finite rank and  $V_0=A M_f=AN_f$.
\end{theorem}
\begin{proof} Consider $V_1=AM$, which is an $A$-subspace of $V$.
Let $ \sigma: V\rightarrow V/V_1$ be the $A$-linear quotient map. Then
$\sigma(N)$ is a $\bfk[t^{-1}]$-submodule of $V/V_1$. Since $ V=M\oplus N$
as a $\bfk$-vector space. Hence, $\sigma(N)=V/V_1$ as $\bfk$-vector spaces. Let $ N_1=N\cap V_1$. Then we have an exact sequence
\[ 0 \rightarrow N_1 \rightarrow N \rightarrow \sigma(N)\rightarrow 0.\]
We cannot claim that this sequence of $ \bfk[t^{-1}]$-modules splits. However,
since $ \sigma $ is $A$-linear, we can choose $\{v_1,\cdots, v_s\}\subseteq N$ such
that $ \{\sigma(v_1), \cdots, \sigma(v_s)\}$ is an $A$-basis of $V/V_1$. We claim that
$ A\{v_1, \cdots, v_{s}\} \subseteq N$. In fact under the $\bfk$-linear
projection map, $p_M: V \rightarrow M$, we have $p_M(A\{v_1, \cdots, v_{s}\} )=0$
since $ V_1\cap A\{v_1, \cdots, v_{s}\}=0$.  We now claim that
$ N=N_1\oplus A\{v_1, \cdots, v_{s}\}$.
It is straightforward to verify that $(M, N_1)$ is in $\rsd(V_1)$.

We know that $ \widetilde{M}\subseteq V_1 $ is an $ A$-subspace of $V_1$. The above corollary \ref{cor:freed} implies that $V_1=\widetilde{M}\oplus V_0$ for an $A$-subspace of $V_1$  such that  $M_0=V_0\cap M$ is an $\calo$-free module of finite rank and $ M=\widetilde{M}\oplus M_0$ as $\calo$-module. It follows from the proof of the corollary \ref{cor:freed}, one can choose $V_0$ such that $ N_1 \subseteq V_0$ and further more, $ V_0=A N_1$. Then $ (M_0, N_1)$ is a regular-singular decomposition for the $A$-vector space $V_0$.
\end{proof}

\section{Classification of finite dimensional regular $RB$-modules}
\label{sec:classify}
\subsection{} In this section, we assume that $ (M,N)$ is in $\rsd(V)$ satisfies the condition that
$M$ is finitely generated $ \calo$-module and $ V=AM$. Let $\rsd_r(V)$ be the set of all such $\rsd$s of $V$.  Under this assumption, $M$ contains an $A$-basis. Since $ M$ is free of finite rank $n=\dim_AV$ over $\calo$ (i.e., $M$ is an $\calo$-lattice in $V$), then there exist $v_{1}, v_{2}, \cdots, v_{n}\in M$ such that $\displaystyle M=\oplus_{i=1}^{n}\calo v_{i}$, and $(v_{1}, v_{2}, \cdots, v_{n})$ is an $A$-basis. Thus $M\cong \calo^n$. From now on, we use this basis  and identify $V=A^n$ and $ \calo^n\subseteq A^n$.

For any lattice $M$ in $V=A^n$, there is $\phi \in GL_n(A)$ such that $ \phi(\calo^n)=M$.  Note that,  for $ \phi \in GL_A(V)$,  $ \phi(\calo^n)=\calo^n$ implies that $\phi \in GL_n(\calo)$.  Thus the set of all lattices in $V$ is in one-to-one correspondence to the affine Grassmannian $ GL_n(A)/GL_n(\calo)$. Thus for each fixed dimension $n$, the map $\Phi: \on{\rsd}_r(V)\rightarrow GL_n(A)/GL_n(\calo) $, defined by $ (M, N)\mapsto \phi GL_n(\calo)$ such that $M=\phi(\calo^n)$, is a fiber bundle and $ GL_n(A)$-equivariant in the sense that $\Phi(\phi(M, N))=\phi\Phi(M,N)$ for all $ \phi\in GL_n(A)$. Our main goal in this section is to compute the fiber of this map.

Let $ \{v_1, \cdots, v_n\}$ be the standard $A$-basis of $V=A^n$. Then $ \calo^n=\oplus_{i=1}^{n}\calo v_i$ and $V/\calo^n$ has a $ \bfk$-basis
$ \{ t^{-j}v_i ~|~ j\geq 1, i=1, 2, \cdots, n\}$. The
$ \bfk$-vector space isomorphism $ N \cong V/\calo^n$ induced via restriction of the
quotient map  $V\rightarrow V/\calo^n$ to $N$ defines a $\bfk$-basis
$\{b_{ji} ~|~ j \geq 1, i=1, 2, \cdots, n\}$ of $N$ such that $b_{ji}=t^{-j}v_i+r_{ji}(t)$ with $ r_{ji}(t) \in \calo^n$.

\begin{lemma} \label{special-basis}
For any $m \in\mathbb{Z}_+$ and $ 1\leq \ell\leq n$, we have
$\displaystyle b_{m\ell}\in \sum_{i=1}^{n}\bfk[t^{-1}]b_{1i}$.
\end{lemma}

\begin{proof} We  use induction on $m$. If $m=1$, it's clear. For $m\geq 2$, assume that $\displaystyle b_{m\ell}\in \sum_{i=1}^n \bfk[t^{-1}]b_{1i}$. Then
\begin{eqnarray*}
b_{(m+1)\ell}&=&t^{-1}b_{m\ell}-t^{-1}r_{m\ell}(t)+r_{(m+1)\ell}(t)\\
&=&t^{-1}b_{m\ell}-t^{-1}r_{m\ell}(0)-t^{-1}(r_{m\ell}(t)-r_{m\ell}(0))+r_{(m+1)\ell}(t).\end{eqnarray*}
Note that $t^{-1}(r_{m\ell}(t)-r_{m\ell}(0))\in \calo^n$.  If we write $\displaystyle r_{m\ell}(0)=\sum_{i=1}^n a_i v_i$ with $a_i \in \bfk$, we have $\displaystyle t^{-1}r_{m\ell}(0)=\sum_{j=1}^n a_ib_{1i}-\sum_{i=1}^n a_i r_{1i}(t)$.
Thus we have $\displaystyle b_{(m+1)\ell} \in \sum_{i =1}^n \bfk[t^{-1}]b_{1i}+\calo^n$.

Write $b_{(m+1)\ell}=u+v$ with $\displaystyle u\in  \sum_{i=1}^n \bfk[t^{-1}]b_{1i} $ and $ v \in \calo^n$.
Since  $\displaystyle \sum_{i=1}^n \bfk[t^{-1}]b_{1i}\subseteq N$ and $ b_{(m+1)\ell}\in N$, thus $v =b_{(m+1)\ell}-u \in N$. Therefore $ v=0$ since $ N\cap \calo^n=\{0\}$. Hence $\displaystyle b_{(m+1)\ell}\in \sum_{i=1}^n \bfk[t^{-1}]b_{1i} $.
 \end{proof}

\begin{prop} \label{o-decomp}
If $(\calo^{n}, N)\in \rsd(V)$, then $N$ is a free $ \bfk[t^{-1}]$-module
of rank $n$ with a basis $\{ b_i~|~ i=1, 2, \cdots, n\}$ where $ b_i \in t^{-1}v_i+\calo^n$.
In particular, there is a $ \phi \in GL_n(\calo)$ such that $ \phi(t^{-1}v_i)=b_i$ for $ i=1, 2, \cdots, n$.  \end{prop}
\begin{proof}
By Lemma~\ref{special-basis},  we have
$\displaystyle N\subseteq \sum_{i=1}^{n} \bfk[t^{-1}]b_{1i} \subseteq N$.
Then $ \{b_{1i}~|~ i=1, \cdots, n\}$ is a $\bfk[t^{-1}]$-basis of $N$.
Write $b_{i}=b_{1i}$. The $A$-linear map $ \phi: A^n\rightarrow A^n$
defined by $ \phi(v_i)=v_i+tr_{1i}(t)$ is invertible and $\phi \in GL_n(\calo)$
since $ \phi(\calo^n)=\calo^n$, and $\displaystyle N=\phi( \sum_{i=1}^n \bfk[t^{-1}]t^{-1}v_i)$.
\end{proof}
\begin{coro} \label{o-orbit} The set of all $ \bfk[t^{-1}]$-submodules $N$ of $ A^n$ such that $ (M, N)\in \rsd(A^n)$ is the  $ GL_n(\calo)$-orbit of $ t^{-1} \bfk[t^{-1}]^n$ which is in bijection to $GL_n(\calo)/GL_{n}(\bfk)$.
\end{coro}
\begin{proof} Clearly for any $\phi\in GL_n(\calo)$, $(\calo^n, \phi(t^{-1}\bfk[t^{-1}]^{n}))$ is also in $\rsd(A^n)$. For any $(\calo^n, N)$, there is $\psi \in GL_n(A)$ such that $ \psi(\calo^n,t^{-1}\bfk[t^{-1}]^{n})=(\calo^n, N)$, we have $\psi(\calo^n)=\calo^n$. Hence $ \psi \in GL_n(\calo)$.  In futher $ \psi(t^{-1}\bfk[t^{-1}]^n) =t^{-1}\bfk[t^{-1}]^n$, then $ \psi \in GL_n(\bfk[t^{-1}])$. Thus $ \psi \in GL_n(\calo)\cap GL_{n}(\bfk[t^{-1}])=GL_n(\bfk)$.
\end{proof}

\begin{theorem} If $ (V, p)$ is a regular $RB$-module for $ (A, P)$ of dimension $n$,
then $ (V, p)\cong (A, P)^{\oplus n}$.
\end{theorem}
\begin{proof} Since $(V, p)$ is regular, we can find an $A$-basis $\{ v_1, \cdots, v_n\}$ such that $\displaystyle p(V)=\sum_{i=1}^n \calo v_i$. By Proposition~\ref{o-decomp}, there is $\phi \in GL_A(V)$ such that $\displaystyle \phi(\sum_{i=1}^n\calo v_i)=\sum_{i=1}^n\calo v_i $ and $(\phi(t^{-1}v_1), \cdots, \phi(t^{-1}v_n)$ is a $\bfk[t^{-1}]$-basis of $\ker (p)$. Thus
$ \phi: (A, P)^{\oplus n}\rightarrow (V, p) $ defines an isomorphism of $RB$-modules over $ (A, P)$.
\end{proof}

\begin{theorem} \label{thm:classification}
Every $RB$-module over $(A, P)$ is semisimple. The category of $RB$-modules has exactly three isomorphism classes of  irreducible $(A, P)$-modules. They  are $A_r=(A, P)$, $A_0=(A, 0)$, and $A_1=(A, Id)$.
\end{theorem}

The following corollary has other geometric applications. The group $ GL_n(A)$ acts on the set $\rsd(A^{\oplus n})$ of all $\rsd$s of $V$. This set carries of structure of infinite dimensional variety, which we will not pursue here.
\begin{coro} \label{coro:classification} The set $RB(A^{\oplus n})$ has exactly $(n+2)(n+1)/2$ many $GL_n(A)$-orbits, which are classified by $ (k, r, l)\in \mathbb{Z}^{3}_{\geq 0}$ with $k+r+l=n$.
\end{coro}

\begin{coro} \label{coro:uniqueness} For any $\mathcal{O}$-submodule $M\subseteq V$, there is a regular-singular decomposition $(M,N)\in \rsd(V)$.  Two $(A, P)$-modules $ (V, p)$ and $ (V', p')$ are isomorphic if and any only there is an $A$-linear isomorphism $\phi: V\rightarrow V'$ such that $ \phi(p(V))=p'(V')$. In particular, $(V,p)$ is uniquely determined by the $\mathcal{O}$-submodule $ p(V)\subseteq V$. 
\end{coro}
\begin{proof} By Corollary~\ref{cor:freed}, $M=\widetilde{M}\oplus M_f$. By Corollary~\ref{coro:classification}, there exists an $N$ such that $ V=\widetilde{M}\oplus V_{0}\oplus \widetilde{N}$. Hence $ (M, N)$ is an RSD of $V$.
\end{proof}
We remark that the linear map $\phi$ in the corollary is not an $RB$-module morphism. However, there is a $\psi \in GL_n(\calo)$ so that $ \psi\phi$ is an isomorphism. 
\section{Orbits and stabilisers}
\subsection{} In this subsection, let $\calo$ be any integral domain and $ A$ be its field of fractions. For any finite dimensional $A$-vector space $V=A^n$ with fixed $ \calo$-lattice $V_\calo=\calo^n$, we consider the set of all $\calo$-submodules $M$ of $V$. $M$ is called a {\em fractional $\calo$-submodule} if there is $\gamma \in GL_n(A)$ such that $ \gamma (M)\subseteq V_\calo$. Clearly $ GL_n(A)$ acts on the set $\calm_n(\calo)$ of all fractional $\calo$-submodules of $A^n$. We are interested in classifying all $GL_n(A)$-orbits $ \calm_n(\calo)/GL_n(A)$. In case $n=1$, the set  $\calm^*_1(\calo)/GL_1(A)=cl(\calo)$ is exactly the ideal classes of $\calo$, with $ \calm^*_1(\calo)$ being the set of all non-zero fractional ideals, and has a monoid structure and $cl(\calo)$ acts on $ \calm_n(\calo)/GL_n(A)$.  When $ \calo$ is a Dedekind domain, then $cl(\calo)$ is an abelian group acting on the set $ \calm_n(\calo)/GL_n(A)$.
\begin{prop} Assume that  $ \calo $ is a split complete discrete valuation ring.  For any $n\geq 1$ we have:
(a) an $\calo$-submodule $M$ is fractional if and only if it is has no nonzero divisible elements;\\
(b) $|\calm_n(\calo)/GL_n(A)|=n+1$ and the orbit of $M$ is completely determined by $\dim_A(AM)$.
\end{prop}
\begin{proof} (a) follows from Theorem 3.3 and the Remark 3.1.  (b) follows from Corollary \ref{cor:fg}.
\end{proof}
Let $ \calm^r_{n}(\calo)$ be the subset of all fractional $\calo$-submodules of  $\calo$-rank r in $A^n$ and $\calg\calr(n,r)$ the set of all $ r$-dimensional $A$-subspaces, then we have a map $\calm^r_{n}(\calo)\rightarrow \calg\calr(n,r)$ with $M\mapsto AM$. This map is $ GL_n(A)$-equivariant with fiber being the $GL_r(A)/GL_r(\calo)$.
Let $\calg _n(\calo)$ be the set of all $ \calo $-submodules of $ A^n$. 

Again $GL_n(A)$ acts on  $\calg_n(\calo)$. Note that $ \calm_n(\calo)$ is a subset of $\calg_{n} (\calo)$.  
By Corollary~\ref{coro:uniqueness}, the map $ \rsd(V)\rightarrow \calg_{n}(\calo)$ is surjective and it induces  a bijection $\rsd(A^{n})/GL_n(A)\rightarrow \calg_n (\calo)/GL_n(A)$.
\begin{coro} When $ \calo$ is a complete discrete valuation ring, then $\calg _n(\calo)/GL_n(A)$ has exactly $(n+2)(n+1)/2$ many elements and each orbit corresponds to exactly one isomorphism class of $n$-dimensional representations of the Rota-Baxter algebra $A$.
\end{coro}

\begin{proof} Using Corollary 3.5, each $M$ is decomposed into $ M=\widetilde{M} \oplus M_f$. The $GL_n(A)$ orbit of $ M$ is completely determined by  $\dim_A\widetilde{M}=r$ and $M_f \subseteq A^{n-r}$.  For each $r$ there exactly $n-r+1$
$GL_{n-r}(A)$-orbits in $ \calm_{n-r}(\calo)$.
\end{proof}

Let $ \Phi_n: \rsd(A^n)\rightarrow \calg_n(\calo)$ such that $\Phi_n(M,N)=M$.  This map is $ GL_n(A) $-equivariant. Each $GL_n(A)$-orbit in $\calg_n(\calo)$ is determined by a pair of non-negative integers $(\dim_A \widetilde{M}, \dim_A (AM_f))=(k,r)$ with $r+f\leq n$. Suppose $M$ is in the orbit corresponding to $(k,r)$, then stabiliser  group $G_M=\{ g\in GL_n(A)\;|\; gM=M\} $ is  (with $\ell=n-k-r$)
\[ G_M\cong \begin{pmatrix} GL_k(A) & 0& 0\\
M_{r,k}(A)&GL_r(\calo)&0\\
M_{\ell,k}(A)&M_{\ell,r}(A)&GL_\ell(A) \end{pmatrix}.\]

Let $\mathfrak{Fl}^3(V)$ be the set of all filtrations of $A$-subspaces $0=V_0\subseteq V_1\subseteq V_2\subseteq V_3=V$. Let $\mathfrak{Fl}_{d_1,d_2,d_3}(A)$ be the subset of  $\mathfrak{Fl}^3(V)$ of all generalised flags of type $(d_1,d_2,d_3)$, i.e, the  filtrations of $ A$-subspaces $ 0=V_0\subseteq V_1\subseteq V_2\subseteq V_3=V$ with $ \dim_A V_{i}/V_{i-1}=d_i$ for $ i=1,2,3$. 
\[\mathfrak{Fl}^l(V) =\dot{\bigcup}_{d_1+d_2+d_3=n}\mathfrak{Fl}_{d_1,d_2,d_3}(A).\]
$GL_n(A)$ acts on $\mathfrak{Fl}^3(V)$ and each $\mathfrak{Fl}_{d_1,d_2,d_3}(A)$ is a $GL_n(A)$-orbit. 
There is a surjective $GL_n(A)$-equivariant map $\rho: \calg_n(\calo)\rightarrow \mathfrak{Fl}^3(V)$ sending the $ \calo$-submodule $M\subseteq V$ to $ V_1 =\tilde{M}\subseteq V_2=A M\subseteq V_3=V$.  Then $ \rho^{-1}(\mathfrak{Fl}_{d_1,d_2,d_3}(A)) $ is a $GL_n(A)$-orbit in $\calg_n(\calo)$.  This gives a complete description of all $ GL_n(A)$-orbits in $\calg_n(\calo)$.
\begin{coro} Let $ \calo=\bfk[[t]]$. For any $M \in \calg_n(\calo)$, the orbit $ GL_n(A)M\cong GL_n(A)/G_M$, which is a fibre bundle over the generalised flag variety $\mathfrak{Fl}_{k,r,\ell}(A)$ of type $ (k, r, \ell)\in \mathbb{Z}_{\geq 0}^{3} $ over the field $A$ with fibre isomorphic to the affine Grassmanian $GL_r(A)/GL_r(\calo)$. In particular, when $ r=n$, the orbit is an affine Grassmanian  $GL_n(A)/GL_n(\calo)$.
\end{coro}

 Note that $G_M$ also acts on the fibre of $\Phi_{n}^{-1}(M)$, we will describe the $G_M$-orbits in $ \Phi_{n}^{-1}(M)$. To do so, we first describe stabiliser $G_{M,N}$.  We recall that the category of $ (A, P)$-modules form an abelian category by \cite{G-Lin}.  Then the stabiliser $ G_{(M,N)}=\{ g \in GL_n(A)\; |\;gM=M, \; gN=N\}$ is isomorphic to the automorphism group the corresponding representation $(A^n, p_M)$ with $ p_M: A^n\rightarrow A^n$ is the projection map onto $M$ with respect to the regular-singular decomposition $ A^n=M\oplus N$ as $ \bfk$-vector space. Thus we compute the automorphism group of each representation.

There are three irreducible representations of the Rota-Baxter algebra $(A, P)$, which are all one dimensional with operator being $0$, $Id$, and $P$ respectively are denoted by
$$
A_{0}:=(A, 0),\quad A_{1}:=(A, Id),\quad A_{r}:=(A, p).
$$
Since $\Hom_{(A, P)}(A, A)$ consists of $A$-linear maps, that is for any $\varphi \in \Hom_{(A, P)}(A, A)$, and any $x \in A$, we have $\varphi (x)=x \varphi (1)$. So $\varphi$ is determined by $\varphi(1)$.

\begin{prop}
\begin{enumerate}
\item \label{it:a}  For any $i, j=0, 1, r$, if $i\neq j$, then $\Hom_{(A, P)}(A_{i}, A_{j})=0$;
\item \label{it:b} $\End_{(A, P)}(A_{0})=\End_{(A, P)}(A_{1})=A$ as $\bfk$-algebra;
\item \label{it:c} $\End_{(A, P)}(A_{r})=\bfk$.
\end{enumerate}
\end{prop}
\begin{proof}
 (\ref{it:a}) It is easy to get $\Hom_{(A, P)}(A_{0},
 A_{1})=\Hom_{(A, P)}(A_{1}, A_{0})=\Hom_{(A, P)}(A_{r}, A_{0})=0$. So we only need to prove $\Hom_{(A, P)}(A_{0}, A_{r})=\Hom_{(A, P)}(A_{1}, A_{r})=\Hom_{(A, P)}(A_{r}, A_{1})=0$.

For any $\varphi \in \Hom_{(A, P)}(A_{0}, A_{r})$ satisfies $P\circ \varphi=\varphi \circ 0$. And for any $x \in A_{0}$, we have
$$P( \varphi (x))=\varphi (0 (x))=0.$$
That is $P(x \varphi (1))=0$, then $x \varphi (1) \in \ker P$ for all $ x \in A_0$. Thus $\varphi(1)=0$.

Similarly, we have $\Hom_{(A, P)}(A_{1}, A_{r})=\Hom_{(A, P)}(A_{r}, A_{1})=0$.

\noindent
(\ref{it:b}) For any $\varphi \in \End_{(A, P)}(A_{0})$ (or $\varphi \in \End_{(A, P)}(A_{1})$), since both maps $ 0, 1: A\rightarrow A$ are $A$-linear and $\varphi$ is also $A$-linear, that is $\varphi(1)$ can be any element of $ A$.

\noindent
(\ref{it:c}) For  $\varphi \in \End_{(A, P)}(A_{r})$, we have $\varphi \circ P=P \circ \varphi$. Note that $P$ is not $ A$-linear. For any $x \in A$, we have $\varphi(x)=x\varphi(1)$. So we only need to compute $\varphi(1)$. From $\varphi(1)=\varphi( P(1))=P( \varphi(1))$, we have  $\varphi(1)\in \calo$, that is $\displaystyle \varphi(1)=\sum_{i=0}^{\infty}c_{i}t^{i}$.

For any $n\in \mathbb{N}_{+}$, $0= \varphi ( P(t^{-n}))=P ( \varphi(t^{-n}))=P(t^{-n}\varphi(1))=P(t^{-n}\sum_{i=0}^{\infty}c_{i}t^{i})=\sum_{i=0}^{\infty}c_{n+i}t^i$. We have $c_{i}=0$ for any $i\geq n>0$. So $\varphi(1)=c\in \bfk$.
\end{proof}

For any $(A, P)$-module $ (A_{0}^{\oplus k}\oplus A_{1}^{\oplus \ell}\oplus A_{r}^{\oplus r})$ we have
\begin{eqnarray*}
\End_{(A, P)}(A_{0}^{\oplus k}\oplus A_{1}^{\oplus \ell}\oplus A_{r}^{\oplus r})&=&\End_{(A, P)}(A_{0}^{\oplus k})\oplus \End_{(A, P)}(A_{1}^{\oplus \ell})\oplus \End_{(A, P)}(A_{r}^{\oplus r})\\
&=&M_{k\times k}(A) \times  M_{\ell\times \ell}(A) \times M_{r \times r}(\bfk).
\end{eqnarray*}

From the Theorem \ref{theorem p}, we have for any $(M, N) \in \rsd(A^{\oplus n})$, there exists an Rota-Baxter presentation operator $p \in RB(A^{\oplus n})$. $GL(A^{\oplus n})$ acts on the set $\rsd(A^{\oplus n})$, now we  determine the stabilizer of $(M, N)$.
\begin{theorem} \label{stabiliser} Let $(M, N)\in \rsd(V)$ be corresponding to $ A_0^{\oplus k}\oplus A_1^{\oplus \ell}\oplus A_r^{\oplus r}$. The
\begin{eqnarray*}
\stab_{GL({A^{\oplus n}})}(M, N)&=&GL_{k}(A)\times GL_{\ell}(A)\times GL_{r}(\bfk).\\
\calo_{k, \ell, r}&=&GL_{n}(A)/(GL_{k}(A)\times GL_{\ell}(A)\times GL_{r}(\bfk)).
\end{eqnarray*}
\end{theorem}

We now return the determining the fibre $ \Phi_{n}^{-1}(M)$ for an $ M\in \calg_n(\calo)$. We recall that $ M=\widetilde{M}\oplus M_f$ is of the type $ (k, r) $ with $ k=\dim_A \widetilde{M}$ and $ r=\dim_A AM=\rank_\calo M_f$. Then $M$ defines an element $ 0\subseteq \widetilde{M} \subseteq AM \subseteq V$  in $ \mathfrak{Fl}_{k,r,\ell}(A)$.
The each element $N$ in the fibre $ \Phi_{n}^{-1}(M)$ is uniquely determined by a pair $ (N_f, \widetilde{N})$. Here $\widetilde{N}$ is an $A$-subspace of $ V$ complement to the $A$-subspace $ AM$, and $ N_f \in \Phi^{-1}_{A^{\oplus r}}(M_f)$. We know the set of the all such $ \widetilde{N}$ form an $A$-variety which is isomorphic to unipotent radical
\[ \begin{pmatrix} I_{k+r} & 0\\
M_{\ell, k+r}(A)& I_\ell
\end{pmatrix}
\]
of the parabolic subgroup $\stab_{GL_n(A)}(AM)$.
On the other hand, by Corollary~\ref{o-orbit},
the fibre of $ \Phi_{A^{r}}^{-1}(\calo^r)$ is a single $ GL_r(\calo)$-orbit.
By Theorem~\ref{stabiliser}, we have
$\stab_{GL_r(\calo)}(N)=\stab_{GL_n(A)}(M_f, N_f)=GL_n(\bfk)$. Hence we have
\[  \Phi_{A^{r}}^{-1}(\calo^r)\cong GL_r(\calo)/GL_r(\bfk).\]

\begin{theorem} The $GL_n(A)$-equivariant map $\Phi_n: \rsd(A^n)\rightarrow \calg_n(\calo)$ is onto. For $M\in \calg_n(\calo)$ of type $ (k, r)$, then fibre of $ \Phi_{n}^{-1}(M)$ is isomorphic $G_M/\aut(M, N)$ which is
\[ \begin{pmatrix} GL_k(A) & 0& 0\\
M_{r,k}(A)&GL_r(\calo)&0\\
M_{\ell,k}(A)&M_{\ell,r}(A)&GL_\ell(A) \end{pmatrix}/\begin{pmatrix} GL_k(A) & 0& 0\\
0&GL_r(\bfk)&0\\
0&0&GL_\ell(A) \end{pmatrix}.\]
\end{theorem}
\begin{remark}
Let $K=GL_n(\bfk[t^-1])$ be the subgroup of $G=GL_n(A)$ and $ GL_n(\bfk[t^{-1}]) $ is the stabilizer of the $k[t^{-1}]$-submodule $ N$ for an RSD $(\calo^n, N)$ in $V$. Set $H=GL_n(\calo)$. The space 
$G/H\times G/K$ is  the set of the  pairs $(M, N)$  with $ M$ being an $\calo$-lattice in $V$ and $ N$ be a $\bfk[t^{-1}]$-lattice. The double cosets $K\ G/H$ is in bijection to the $G$-orbits in $G/H\times G/K $. This spaces are important subject of study  in Lie representation theory. $RSD(V)$ is the $G$-orbit of $(K, H)$ in $G/K\times G/H$. Or it corresponds to the double coset $KH$ in $G$. It would be interesting to describe all double cosets $G$ as a certain space of representations of the Rota-Baxter algebra $(A, P)$ of dimension $n$. This will be discussed in a future paper. 
\end{remark}

\end{document}